\documentclass[reqno,b5paper]{amsart}
\usepackage{amssymb}
\usepackage{dsfont}
\usepackage{color}

\newtheorem{theorem}{Theorem}[section]
\newtheorem{problem}[theorem]{Problem}
\newtheorem{proposition}[theorem]{Proposition}

\newtheorem{lemma}[theorem]{Lemma}

\newtheorem{remark}[theorem]{Remark}
\newtheorem{definition}[theorem]{Definition}
\newtheorem{fact}[theorem]{Fact}

\newcommand{\T}{\mathbb{T}}
\newcommand{\Z}{\mathbb{Z}}
\newcommand{\N}{\mathbb{N}}

\def\cont{\mathfrak c}

\newcommand{\supp}{\mathrm{supp}}

\catcode`\@=12
\def\T{{\mathbb T}}
\def\eps{{\varepsilon}}

\def\Z{{\mathbb Z}}
\def\N{{\mathbb N}}
\def\R{{\mathbb R}}
\def\Q{{\mathbb Q}}
\def\P{{\mathbb P}}

\def\cont{\mathfrak c}
\def\c{\cont}

\makeindex

\allowdisplaybreaks
\begin{document}
\title[$s$-characterized subgroups]{Statistically characterized subgroups related to arithmetic-type sequence of integers}
\subjclass[2010]{Primary: 22B05, Secondary: 11B05, 40A05} \keywords{Circle group, characterized subgroup, natural density, s-characterized subgroup, arithmetic-type sequence, lifting function}

\author{Pratulananda Das}
\address{Department of Mathematics, Jadavpur University, Kolkata-700032, India}
\email{pratulananda@yahoo.co.in}

\author{Ayan Ghosh}
\address{School of Mathematical and Computational Sciences, Indian Association for the Cultivation of Science, Kolkata-700032, India}
\email {ayanghosh.jumath@gmail.com}

\author{Tamim Aziz}
\address{Department of Mathematics, University of North Bengal, Darjeeling-734013, India}
\email {tamimaziz99@gmail.com}

\begin{abstract}
Very recently, in [Das et al., J. Lond. Math. Soc., 2025], statistically characterized subgroups were studied for certain classes of non-arithmetic sequences. Subsequently, in [Das et al., Bull. Sci. Math., 2025], characterized subgroups were investigated for a class of arithmetic-type sequences that includes both arithmetic sequences and certain non-arithmetic sequences. Motivated by these developments, we study statistically characterized subgroups associated with a broader class of arithmetic-type sequences. In particular, all previously obtained cardinality related observations for statistically characterized subgroups corresponding to arithmetic sequences as well as certain non-arithmetic sequences follow as special cases of our results. Moreover, we show that this broader class exhibits drastically different behavior and differs significantly from the previously studied special cases.
\end{abstract}
\maketitle

\section{Introduction and background}
The concept of characterized subgroups has evolved significantly over the years, broadening its scope beyond its original foundations to serve as a generalization of the torsion subgroup. Its rich history is closely tied to the study of sequences of multiples of a real number mod 1. The classical case, where these sequences are uniformly distributed mod 1, was first explored in the seminal work of Herman Weyl \cite{W}, and has since inspired further developments (see \cite{BDBW, BDMW1, DPS, DGDis} for more details and recent advancements). Moreover, these sequences of multiples have a deep connection with Ergodic theory, as can be seen from the study of Sturmian sequences and Hartman sets \cite{Wi}. The concept also plays a pivotal role in the structure theory of locally compact abelian groups \cite{AT,AD,R}. Furthermore, there is a profound relationship between characterized subgroups and ``trigonometric thin sets", particularly Arbault sets \cite{A1}, which are significant in Harmonic Analysis (see \cite{BuKR, DG, El, Ka} for detailed investigations in these directions).

Before proceeding further, let us formally present the definition of a {\em characterized subgroup} of the circle group $\T$.

\begin{definition}
Let $(a_n)$ be a sequence of integers, the subgroup
$$
t_{(a_n)}(\T) := \{x\in \T: a_nx \to 0\mbox{ in } \T\}
$$
of $\T$ is called {\em a characterized} $($by $(a_n))$ {\em subgroup} of $\T$.
\end{definition}

The term {\em characterized} appeared much later, coined in \cite{BDS}. Further, it is important to note that, for practical purposes, it is sufficient to work with sequences of positive integers only, which has been the usual practice. Although these subgroups can be generated by any sequence of positive integers, one major theme over the years has been the study of these subgroups when they are generated by arithmetic sequences. Precisely, a sequence of positive integers $(a_n)$ is an
{\em arithmetic sequence} if
$$1 = a_0 < a_1 < a_2 < \dots < a_n < \dots ~~\mbox{and}~a_n|a_{n+1}~ \mbox{for every}~n \in \mathbb{N}.$$
For an arithmetic sequence of integers $(a_n)$, the sequence of ratios $(b_n)$ is defined as
$$
b_1=a_1 \mbox{     \ and \      } b_n=\frac{a_n}{a_{n-1}}  \mbox{ for  } n \geq 2.
$$

In 2020, the notion of characterized subgroups was generalized using a more general mode of convergence and here the idea of natural density (see \cite{Bu1}) and the corresponding mode of convergence came into the picture. In the language of \cite{DDB}, ``Although the correspondence $(a_n) \to t_{(a_n)}(\T)$ is decreasing (with respect to inclusion), in many cases the subgroup $t_{(a_n)}(\T)$ is rather small, even if the sequence $(a_n)$ is not too dense". So it seemed a natural course of action to involve a more general mode of convergence (for more details of the reasons and motivation behind this approach see \cite{DDB}).

For $m,n\in\mathbb{N}$ and $m\leq n$, let $[m, n]$ denote the set $\{m, m+1, m+2,...,n\}$. By $|A|$ we denote the cardinality of a set $A$. The lower and upper natural densities of $A \subseteq \mathbb{N}$ are defined by
$$
\underline{d}(A)=\displaystyle{\liminf_{n\to\infty}}\frac{|A\cap [0,n-1]|}{n} ~~\mbox{and}~~
\overline{d}(A)=\displaystyle{\limsup_{n\to\infty}}\frac{|A\cap [0,n-1]|}{n}.
$$
If $\underline{d}(A)=\overline{d}(A)$, we say that the natural density of $A$ exists and it is denoted by $d(A)$. As usual,
$$
\mathcal{I}_d = \{A \subseteq \mathbb{N}: d(A) = 0\}
$$
denotes the ideal of ``natural density zero" sets and $\mathcal{I}_d^*$ is the dual filter, i.e., $\mathcal{I}_d^* = \{A \subseteq \mathbb{N}: d(A) = 1\}$. Let us now recall the notion of statistical convergence in the sense of \cite{F,Fr,S,St,Z}.

\begin{definition}\label{Def1}
A sequence of real numbers $(x_n)$ is said to {\em converge} to a real number $x_0$ {\em statistically} if for any $\eps > 0$, $d(\{n \in \mathbb{N}: |x_n - x_0| \geq \eps\}) = 0$.
\end{definition}

It is now known \cite{S} that $x_n \to x_0$ statistically precisely when there exists a subset A of $ \N$ of asymptotic density 0, such that $\displaystyle{\lim_{n \in \N \setminus A}} x_n = x_0$ which makes this convergence interesting but not too wild. The concept of statistical convergence has been extensively studied over the years, beginning in metric spaces and later expanding to general topological spaces \cite{MK}. Over the past three decades, significant progress has been made in this area, largely due to its natural extension of the notion of usual convergence. Statistical convergence retains many fundamental properties of traditional convergence, while encompassing a broader class of sequences. For some insightful applications of statistical convergence, the readers are referred to \cite{BDK,CKG,FGT}. As a natural application, the following notion was introduced in \cite{DDB}:
\begin{definition}
For a sequence of integers $(a_n)$ the subgroup
\begin{equation}\label{def:stat:conv}
t^s_{(a_n)}(\T) := \{x\in \T: a_nx \to 0\  \mbox{ statistically in }\  \T\}
\end{equation}
of $\T$ is called {\em a statistically characterized} (shortly, {\em an s-characterized}) $($by $(a_n))$ {\em subgroup} of $\T$.
\end{definition}

It has later come to our notice that the notion of statistically characterized subgroup actually appeared without any name or in the above form in the work of Borel (see \cite{B1} where it was denoted by $G^*(\Lambda)$). The subsequent result justifies the investigation of this new notion of s-characterized subgroups, as it turns out that, though in general larger in size, these subgroups are still essentially topologically nice.

\begin{theorem}\label{theoremA}\cite[Theorem A]{DDB}
For any increasing sequence of integers $(a_n)$, the subgroup $t^s_{(a_n)}(\T)$ is a $F_{\sigma\delta}$ (hence, Borel) subgroup of $\T$ containing $t_{(a_n)}(\T)$.
\end{theorem}

The following theorem leads to a general assertion about the size of s-characterized subgroups for arithmetic sequences.

\begin{theorem}\label{theoremB}\cite[Theorem B]{DDB}
Let $(a_n)$ be an arithmetic sequence. Then $|t^s_{(a_n)}(\T)|=\mathfrak c$.
\end{theorem}

As a consequence, we find the more general result that the new subgroup $t^s_{(a_n)}(\T)$ always
differs from the subgroup $t_{(a_n)}(\T)$ for arithmetic sequences.

\begin{theorem}\label{theoremC}\cite[Theorem C]{DDB}
For any arithmetic sequence $(a_n)$, $t^s_{(a_n)}(\T)\neq t_{(a_n)}(\T)$.
\end{theorem}

Unlike Theorem A, both Theorem B and Theorem C rely on  arithmetic sequence. Motivated by these observations, a series of open problems were posed in \cite{DDB,DDBH}.
\begin{problem}\label{prob1}\cite[Question 6.3]{DDB}
Do Theorems B and C hold true for arbitrary sequences $(u_n)$?
\end{problem}
\begin{problem}\label{prob2}\cite[Question 7.16]{DDBH}
Does there exist an infinite statistically characterized subgroup that can be characterized?
\end{problem}
\begin{problem}\label{prob3}\cite[Question 7.17]{DDBH}
Does there exist an increasing sequence $(u_n)$ such that $t^s_{(u_n)}(\T)$ is countably infinite?
\end{problem}

Coming back to the literature regarding characterized subgroups, in an interesting departure, a non-arithmetic sequence $(\zeta_n)$ was defined in \cite{DK} as follows:
\begin{equation}\label{eqnonarith}
1,2,4,6,12, 18, 24,  \ldots, n!, 2\cdot n!, 3 \cdot n!, \ldots , n \cdot n!, (n+1)!, \ldots
\end{equation}
It was established in \cite{DK} that $t_{(\zeta_n)}(\T) = \Q/\Z$. Motivated by this observation, for an arithmetic sequence $(a_n)$ the following general class of non-arithmetic sequences was introduced in \cite{DG8}. Let $(d_n^{a_n})$ be an increasing sequence of integers formed by the elements of the set,
\begin{equation}\label{nonarithdef}
\{r_{k+1}a_k \ : \ 1\leq r_{k+1}< b_{k+1}\}.
\end{equation}
We simply denote $(d_n^{a_n})$  by $(d_n)$ when there is no confusion in the indexing. Note that for $a_n=n!$ corresponding non-arithmetic sequence $(d_n)$ coincides with the sequence $(\zeta_n)$. A through investigation regarding the subgroup $t_{(d_n)}(\T)$, in particular cardinality and structure related questions and its relation with $t_{(a_n)}(\T)$ was carried out in \cite{DG8}. It is to be noted that $(a_n)$ is a subsequence of $(d_n)$. The natural question as to what happens when one considers an arbitrary subsequence of $(d_n)$ containing $(a_n)$ was carried out in \cite{DGT1}. To be more specific, we are interested here in those increasing sequences of integers $(e_n)$ (which will be called ``arithmetic-type" sequences) satisfying
\begin{equation}\label{seqarithtype}
(a_n)\subseteq (e_n^{a_n}) \subseteq (d_n^{a_n}).
\end{equation}
Again, when there is no confusion regarding the sequence $(a_n)$, we simply denote this sequence by $(e_n)$.
\begin{remark}\label{uconr1}
Since $(a_n)$ is a subsequence of $(e_n)$, we can write $a_k=e_{n_k}$. More precisely, we can write $e_{(n_{k-1}+i-1)}=r^i_ka_{k-1}$ where $1\leq i \leq n_{k}-n_{k-1}$ and $r_k^i\in [1,q_{k}-1]$ with $r^1_k=1$. In particular, if $(e_n)=(d_n)$ then $n_{k}-n_{k-1}=q_{k}-1$ and $d_{(n_{k-1}+i-1)}=ia_{k-1}.$ We simply denote $r_k^i$ by $r_k(e)$ when there is no confusion in the indexing.
\end{remark}
\begin{definition}\label{ucondeff2}
For a strictly increasing sequence $(n_k)$ of natural numbers, let us define the {\em lifting function} $L_{(n_k)} \ : \ \mathcal{P}(\N) \to \mathcal{P}(\N)$ as
$$
L_{(n_k)}(A)=\bigcup\limits_{k\in A} [n_{k-1},n_{k}-1].
$$
\end{definition}
When there is no confusion regarding the sequence $(n_k)$, we simply denote this lifting function by $L$. Based on this, we now introduce the three most important definitions, which will be used time and again in the sequel and will help us formulate our main results.
\begin{definition}
An arithmetic-type sequence $(e_n)$ corresponding to the arithmetic sequence $(a_n)$ is called {\em $L$-invariant} if for every infinite $A\subseteq\N$ with $d(A)=0$ implies ${d}(L(A))=0$.
\end{definition}
\begin{definition}
An arithmetic-type sequence $(e_n)$ corresponding to the arithmetic sequence $(a_n)$ is called {\em weakly $L$-invariant} if there exists an infinite $A\subseteq\N$ such that ${d}(L(A-m))=0$ for each $m\in\N\cup \{0\}$.
\end{definition}
\begin{definition}
An arithmetic-type sequence $(e_n)$ corresponding to the arithmetic sequence $(a_n)$ is called {\em strongly non $L$-invariant} if for any $b$-divergent $A\subseteq\N$, $\bar{d}(L(A))>0$.
\end{definition}
As a natural consequence of the works done in \cite{DG8} and \cite{DDB}, the subgroup $t^s_{(d_n)}(\T)$ was considered in \cite{DG9}. One of the main results of \cite{DG9} was the following interesting observation, which provides a negative solution to Question 6.6 posed in \cite{DDB}.
\begin{theorem}\label{sconth}\cite[Theorem 2.14]{DG9}
Let $(a_n)$ be an arithmetic sequence such that for each $m\in\N$, $\lim\limits_{n\to\infty}\frac{\sum\limits_{i=0}^{m-1} (b_{n-i}-1)}{\sum\limits_{i=1}^n (b_i-1)}=0$. Then $|t^s_{(d_n)}(\T)|=\mathfrak{c}$.
\end{theorem}
Motivated by this observation, the following open problem was posed in \cite{DG9}.
\begin{problem}\label{prob4}\cite[Problem 2.16]{DG9}
Is $|t^s_{(d_n)}(\T)|=\mathfrak{c}$, for an arbitrary arithmetic sequence of integers $(a_n)$?
\end{problem}
We now present the most interesting observations carried out in \cite{DGT2} using terminologies adopted in this article, which provide solutions to Problem \ref{prob1}, Problem \ref{prob2}, Problem \ref{prob3} and Problem \ref{prob4}.
\begin{theorem}\label{sconthmain01}
If $(d_n)$ is weakly $L$-invariant, then $|t^{s}_{(d_n)}{(\mathbb{T})}|=\mathfrak{c}$.
\end{theorem}
\begin{theorem}\label{sconcoromain1}
If $(d_n)$ is weakly $L$-invariant then $t^{s}_{(d_n)}{(\mathbb{T})}\neq t_{(d_n)}(\mathbb{T})$.
\end{theorem}
\begin{theorem}\label{sconthmain2}
Suppose $(d_n)$ is not weakly $L$-invariant. Then $t_{(d_n)}^s(\T)=t_{(d_n)}(\T)$ if and only if $(d_n)$ is strongly non $L$-invariant.
\end{theorem}
In this article, our primary focus is on the ``statistically characterized subgroups" generated by the sequences $(e_n)$, i.e., $t^s_{(e_n)}(\T)$ and to understand the cardinality pattern of the statistically characterized subgroup generated by the sequence $(e_n)$. In Theorem \ref{sconthmain1}, we have been able to show that for weakly $L$-invariant arithmetic-type sequences, they are always of size $\mathfrak{c}$. This in turn provides a more general view of Theorem \ref{theoremB}, Theorem \ref{sconth} and Theorem \ref{sconthmain01}. Further they are always bigger than the corresponding characterized subgroup which also provides more general view of Theorem \ref{theoremC} and Theorem \ref{sconcoromain1}. Another main observation states that if the underlying arithmetic-type sequence is not strongly non $L$-invariant, then also they have cardinality $\mathfrak{c}$. Now, note that if the arithmetic-type sequence $(d_n)$ is not weakly $L$-invariant and also strongly non $L$-invariant then Theorem \ref{sconthmain2} ensures that the cardinality of the subgroup $t_{(d_n)}^s(\T)$ is countable. We conclude this article with the observation ensuring that this result does not hold in general for the broader class of arithmetic-type sequences.
\section{Notation and terminology.\vspace{.3cm} \\ }
Throughout $\R$, $\Q$, $\Z$, $\P$ and $\N$ will stand for the set of all real numbers, the set of all rational numbers,
the set of all integers, the set of primes and the set of all natural numbers (note that we do not consider zero as a natural number) respectively. The first three are equipped with their usual abelian group structure and the circle group $\T$ is identified with the quotient group $\R/\Z$ of $\R$ endowed with its usual compact topology.

Following \cite{Ka}, we may identify $\T$ with the interval [0,1] identifying 0 and 1. Any real valued function $f$ defined on $\T$ can be identified with a periodic function defined on the whole real line $\R$ with period 1, i.e., $f(x+1)=f(x)$ for every real $x$. When referring to a set $X\subseteq \T$ we assume that $X\subseteq [0,1]$ and $0\in X$ if and only if $1\in X$. For a real $x$, we denote its fractional part by $\{x\}$ and the distance from the integers by $\|x\|=\min\big\{\{x\},1-\{x\}\big\}$.

For arithmetic sequences, the following facts (see \cite{DG9}) will be used in this sequel time and again. So, before moving onto our main results here we recapitulate that once.
\begin{fact}\label{lemmanew}\cite{DI1}
For any arithmetic sequence $(a_n)$ and $x\in [0,1]$, we can find a unique sequence $c_n\in [0,b_n-1]$ such that
\begin{equation}\label{canonical:repr}
x=\sum\limits_{n=1}^{\infty}\frac{c_n}{a_n},
\end{equation}
where $c_n<b_n-1$ for infinitely many $n$.
\end{fact}
%

For $x\in\T$ with canonical representation (\ref{canonical:repr}), we define
\begin{itemize}
\item[$\bullet$] $supp(x) = \{n\in \N: c_n \neq 0\}$,
\item[$\bullet$] $supp_b(x)=\{n\in\N\ : \ c_n=b_n-1\}$.
\end{itemize}

Note that, $A\subseteq\N$ is called
\begin{itemize}
\item[(i)] $b$-bounded if the sequence of ratios $(b_n)_{n\in A}$ is bounded.
\item[(ii)] $b$-divergent if the sequence of ratios $(b_n)_{n\in A}$  diverges to $\infty$.
\end{itemize}
We say $(a_n)$ is $b$-bounded ($b$-divergent) if $\N$ is $b$-bounded ($b$-divergent). Also, note that for each $j\in\N$,
\begin{equation}\label{eqsum}
\sum\limits_{i=j}^\infty \frac{c_i}{a_i} \leq \sum\limits_{i=j}^\infty \frac{b_i-1}{a_i} = \sum\limits_{i=j}^\infty \bigg(\frac{1}{a_{i-1}} - \frac{1}{a_i} \bigg) \leq \frac{1}{a_{j-1}}.
\end{equation}
\begin{lemma}\label{uconlmain}\cite[Lemma 3.1]{DR}
For $x\in[0,1]$ with canonical representation (\ref{canonical:repr}), for every natural $n > 1$ and every non-negative integer $t$,
\begin{equation}\label{eqlemain1}
\{a_{n-1}x\}=\frac{c_n}{b_n}+\frac{c_{n+1}}{b_nb_{n+1}}+\ldots+\frac{c_{n+t}}{b_n\ldots b_{n+t}}+\frac{\{a_{n+t}x\}}{b_n\ldots b_{n+t}}.
\end{equation}
In particular, for $t=1$, we get
\begin{equation}\label{eqlemain2}
\{a_{n-1}x\}=\frac{c_n}{b_n}+\frac{c_{n+1}}{b_nb_{n+1}}+\frac{\{a_{n+1}x\}}{b_nb_{n+1}}.
\end{equation}
\end{lemma}
\begin{fact}\label{eqInorm}
The following facts are well known for integer norm:
\begin{itemize}
\item for any $A\subseteq\N$ and $y_n=x_n+z_n$ if $\lim\limits_{n\in A} z_n=0$ then $\lim\limits_{n\in A}\|y_n\|=\lim\limits_{n\in A}\|x_n\|$.
\item for any integer $n$ and for any real number $x$, $\|n+x\|=\|x\|$.
\end{itemize}
\end{fact}
Also note that for any $r\in\N$,
\begin{equation}\label{eqr}
\bullet \ \mbox{if } \ r\{a_nx\}<1\ \mbox{ then } \ \{ra_nx\}=r\{a_nx\}.
\end{equation}
\begin{equation}\label{eqrn}
\bullet \ \mbox{if } \ r\|a_nx\|<\frac{1}{2} \ \mbox{ then } \ \|ra_nx\|=r\|a_nx\|.
\end{equation}
\begin{remark}\label{uconr01}
Observe that $L$ is injective and for a sequence $(A_n)$ of sets in $\mathbb{N}$, we can write
\begin{itemize}
\item $L(\displaystyle\bigcup_{n=1}^{\infty}A_n)=\displaystyle\bigcup_{n=1}^{\infty}L(A_n)$ and $L(\displaystyle\bigcap_{n=1}^{\infty}A_n)=\displaystyle\bigcap_{n=1}^{\infty}L(A_n)$,
\item $L(A_m\setminus A_n)=L(A_m)\setminus L(A_n)$ for all $m,n\in\mathbb{N}$.
\end{itemize}
\end{remark}
For $A\subseteq\mathbb{N}$ and $m\in\mathbb{N}\cup \{0\}$, we write
$$A-m=\{n\in\mathbb{N}:n=a-m~\mbox{ for some }~a\in A\}.$$
\vspace{.1cm}

\section{Main results.\vspace{.3cm} \\ }

\begin{lemma}\label{le$L$-invariantqboundnew}   Let $(e_n)$ be an arithmetic-type sequence. Then for any $b$-bounded $A\subseteq\N$, $d(A)=0$ implies $d(L(A))=0$.
\end{lemma}
\begin{proof}
Let $(e_n)$ be an arithmetic-type sequence corresponding to the arithmetic sequence $(a_n)$.
Suppose $A$ is a $b$-bounded subset of $\N$ such that $d(A)=0$. Then there exists $M\in\mathbb{N}$ such that $2\leq b_n\leq M$ for each $n\in A$. Now Remark \ref{uconr1} entails that for each $k\in A$
$$
a_k=e_{n_k}~\mbox{ and }~ n_{k}-n_{k-1}\leq b_{k}-1<M.
$$
Let us set $A'=\{n_{k-1}:k\in A\}$. It is clear that $n_k\geq k$ for each $k\in\mathbb{N}$. Therefore, observe that
\begin{align*}
	\bar{d}(A')&=\limsup_{n\to\infty}\frac{|\{j\in\mathbb{N}:j\leq n~\mbox{and}~j\in A'\}|}{n}\\
	 &\leq \limsup_{n\to\infty}\frac{|\{j\in\mathbb{N}:j\leq n~\mbox{and}~j\in A\}|}{n}\\
	&=\bar{d}(A)=0.
\end{align*}
It is easy to realize that $L(A)\subseteq\displaystyle\bigcup_{i=0}^{M}(A'+i)$. Since $d(A')=0$, we obtain $d(A'+i)=0$ for each $i\in\mathbb{N}$. This also ensures that $d(L(A))=0$.
\end{proof}
\begin{proposition}\label{pro$L$-invariantqbound}
Any arithmetic-type sequence corresponding to a fixed $b$-bounded arithmetic sequence is $L$-invariant, but the converse does not hold.
\end{proposition}
\begin{proof}
Assume that $(e_n)$ is an arithmetic-type sequence associated with the
$b$-bounded arithmetic sequence $(a_n)$. Pick any $A\subseteq\mathbb{N}$ such that $d(A)=0$. As $A$ is again $b$-bounded, Lemma \ref{le$L$-invariantqboundnew} ensures that $d(L(A))=0$. Consequently,  $(e_n)$ is $L$-invariant.\\

To illustrate that the converse does not necessarily hold, we construct an arithmetic sequence $(a_n)$ that is not $b$-bounded, yet every arithmetic-type sequence $(e_n)$ corresponding to $(a_n)$ is $L$-invariant.\\
We construct a set \( D \subseteq \mathbb{N} \) such that
$$
D=\bigcup_{j=1}^{\infty}[l_j,m_j]
$$
where $l_1=1, l_j\leq m_j,|l_{j+1}-m_j|\to\infty$, and $d(D)=1$ (note that the existence of such a set can always be guaranteed; for instance, one may choose $m_j - l_j = j^3$ and $l_{j+1} - m_j = j$
). Let us rewrite
$$
K=\{n_0\ (=1)<n_1<n_2<...<n_k<...\}.
$$
From the construction of $D$, one can obtain a sequence $(s_j)$ such that $n_{s_j}=m_j$ and $n_{{s_j}+1}=l_{j+1}$.
For each $k\in\mathbb{N}\cup\{0\}$, we set
$$
b_{k+1}=n_{k+1}-n_{k}+1.
$$
Note that
\begin{equation*}
b_k =
\left\{
\begin{array}{lr}
l_{j+1}-m_{j}+1&\mbox{ if }~k=s_{j}+1,\\
2&\mbox{otherwise}.
\end{array}
\right.
\end{equation*}
Thus, the associated arithmetic sequence $(a_n)$ is given by
\begin{equation*}
    a_0=1~\mbox{ and }~a_{n+1}=b_{n+1}a_n.
\end{equation*}
Since $(b_{s_j+1})$ is divergent, it is evident that $(a_n)$ is not $b$-bounded.

Let $(e_n)$ be any arithmetic-type sequence corresponding to $(a_n)$. Now, let us pick any infinite $A\subseteq\mathbb{N}$ with $d(A)=0$. If $A$ is $b$-bounded then from Lemma \ref{le$L$-invariantqboundnew} it follows that $d(L(A))=0$. So, let us assume that $A$ is not $b$-bounded. We set
$$
B=\big\{n\in A : n\notin \{s_j+1: j\in\N\}\big\},
$$
and,
$$
I=\big\{n\in A : n\in \{s_j+1: j\in\N\}\big\}.
$$
From the construction of $(b_n)$, one can observe that $B,I\subseteq\mathbb{N}$ with $B$ being $b$-bounded and $I$ being $b$-divergent such that $A=B\cup I$. Since $B\subseteq A$ and $d(A)=0$ we already have $d(B)=0$. Therefore, from Lemma \ref{le$L$-invariantqboundnew}, it is evident that $d(L(B))=0$. Since $I\subseteq\{s_j+1:j\in\mathbb{N}\}$, we also have
$$
L(\{s_j+1:j\in\mathbb{N}\})\subseteq \displaystyle\bigcup_{j=1}^{\infty}[m_j,l_{j+1}-1] \subseteq (\N\setminus D)\cup \{m_j:j\in\mathbb{N}\},
$$
and consequently,
\begin{eqnarray*}
\bar{d}(L(I)) \leq \bar{d}(L(\{s_j+1:j\in\mathbb{N}\})) &\leq& \bar{d}((\mathbb{N}\setminus D) \cup \{m_j:j\in\mathbb{N}\}) \\ &\leq& \bar{d}(\mathbb{N}\setminus D) +\bar{d}(\{m_j:j\in\mathbb{N}\}) \\ &=& \bar{d}(\{m_j:j\in\mathbb{N}\})\ \mbox{ (since $d(D)=1)$} \\ &=& 0 \ \mbox{ (since $|m_{k+1}-m_{k}|\to\infty$)}.
\end{eqnarray*}
Finally, in view of Remark \ref{uconr01}, we can conclude that $d(L(A))=0$. Thus, $(e_n)$ is $L$-invariant.  Since $(e_n)$ is chosen arbitrarily, we conclude that every arithmetic-type sequence corresponding to $(a_n)$ is $L$-invariant.
\end{proof}

\begin{proposition}\label{pro$L$-invariantw$L$-invariant}
Any $L$-invariant arithmetic-type sequence is weakly $L$-invariant.
\end{proposition}
\begin{proof}
Let $(a_n)$ be an arithmetic sequence and $(e_n)$ any corresponding $L$-invariant arithmetic-type sequence.
 Fix an infinite set $A \subseteq \mathbb{N}$ such that $d(A)=0$.  Since $d$ is translation invariant, $d(A-m)=0$ for each $m\in \mathbb{N}\cup \{0\}$.
 As $(e_n)$ is $L$-invariant, it follows that ${d}(L(A-m))=0$ for each $m\in\mathbb{N}\cup \{0\}$.  Consequently, $(e_n)$ is weakly $L$-invariant.
\end{proof}

\begin{proposition}\label{prow$L$-invariantqbound}
Let $(e_n)$ be an arithmetic-type sequence corresponding to the arithmetic sequence $(a_n)$ with $e_{n_k}=a_k$ for each $k\in\N\cup\{0\}$. If $\lim\limits_{k\to\infty}\frac{ n_{k}-n_{k-1}}{n_k-n_0}=0$ then $(e_n)$ is weakly $L$-invariant.
\end{proposition}

\begin{proof}
Assume that $(e_n)$ is an arithmetic-type sequence corresponding to $(a_n)$ such that  $e_{n_k}=a_k$ for each $k\in\N\cup\{0\}$, and  \begin{equation}\label{wdl condition}
 \lim\limits_{k\to\infty}\frac{ n_{k}-n_{k-1}}{n_k-n_0}=0.
 \end{equation}
We now define the sequence $(u_n)$ recursively as follows:
	$$u_1=1~\mbox{ and }~u_{j+1}=\min\{k\in\mathbb{N}:~ k>u_{j}+j+1~\mbox{ and }~ n_{k}>j\sum_{i=1}^{j}\sum_{t=0}^{i} (n_{u_i+1-t}-n_{u_i-t})\}.$$
 From the construction of $(u_n)$, it is evident that
 \begin{equation} \label{Eq 12}
 n_{u_{j-1}}>(j-2)\displaystyle\sum_{i=1}^{j-2}\sum_{t=0}^{i} (n_{u_i+1-t}-n_{u_i-t})~\mbox{ for each }~j\in\mathbb{N}\setminus\{1\}.
  \end{equation}
Set $A=\{u_j+1:j\in\mathbb{N}\}$. We claim that $d(L(A-m))=0$ for each $m\in\mathbb{N}\cup\{0\}$. Choose $m\in\mathbb{N}\cup\{0\}$. Then it is clear that $u_j>u_{j-1}+m+1~$ whenever $~j> m$. Let us set
 $$
 B_m=\displaystyle\bigcup_{j=m+1}^{\infty}[n_{u_j-m},n_{{u_j-m}+1}-1].
 $$
First, observe that $L(A-m)\subset^* B_m$. Next note that
 \allowdisplaybreaks
 \begin{eqnarray*}
 \bar{d}(B_m)&=&\limsup_{n\to\infty}\frac{|B_m\cap[1,n]|}{n}\\
 &=&\limsup_{j\to\infty}\frac{|B_m\cap[1,n_{{u_j-m}+1}-1]|}{n_{{u_j-m}+1}-1}\\
  &\leq&\limsup_{j\to\infty}\frac{|B_m\cap[1,n_{u_{j-2}-m+1}-1]|}{n_{{u_j}-m+1}-1} \\ & &+ \limsup_{j\to\infty}\frac{|B_m\cap[n_{u_{j-1}-m}, n_{{u_{j-1}}-m+1}-1]|}{n_{{u_j}-m+1}-1}\\
  & &+ \limsup_{j\to\infty}\frac{|B_m\cap[n_{u_j-m}, n_{{u_j}-m+1}-1]|}{n_{{u_j}-m+1}-1}\\
 &\leq&\limsup_{j\to\infty}\frac{\displaystyle\sum_{i=m+1}^{j-2} (n_{u_i-m+1}-n_{u_i-m})}{n_{{u_{j-1}}}}+\limsup_{j\to\infty}\frac{(n_{u_{j-1}-m+1}-n_{u_{j-1}-m})}{\bigg(n_1-1+\displaystyle\sum_{i=2}^{u_j-m+1}(n_i-n_{i-1})\bigg)}
 \\
  & &+\limsup_{j\to\infty}\frac{(n_{u_{j}-m+1}-n_{u_{j}-m})}{\bigg(n_1-1+\displaystyle\sum_{i=2}^{u_j-m+1}(n_i-n_{i-1})\bigg)}\\
  &\leq& \limsup_{j\to\infty} \frac{\displaystyle\sum_{i=1}^{j-2}\sum_{t=0}^{i} (n_{u_i+1-t}-n_{u_i-t})}{n_{u_{j-1}}}+\limsup_{j\to\infty}\frac{n_{u_{j-1}-m+1}-n_{u_{j-1}-m}}{n_{u_{j-1}-m+1}-n_{0}}\\ & & +\limsup_{j\to\infty}\frac{n_{u_{j}-m+1}-n_{u_{j}-m}}{n_{u_j-m+1}-n_{0}}\\
&\leq& \lim_{j\to\infty}\frac{1}{j-2}+0~\mbox{ (in view of Eq (\ref{wdl condition}) and Eq (\ref{Eq 12}))}\ =\ 0.
 \end{eqnarray*}
 Consequently, $d(B_m)=0$ and this ensures that $d(L(A-m))=0$ since $L(A-m)\subset^* B_m$. As \( m \in \mathbb{N}\cup\{0\} \) was chosen arbitrarily, we can deduce that $(e_n)$ is weakly $L$-invariant.
\end{proof}
\let\thefootnote\relax\footnotetext{For any two subsets $A$ and $B$ of $\mathbb{N}$ we will denote $A\subset^*B$ if $A\setminus B$ is finite and $A=^*B$ if $A\Delta B$ is finite.}

\begin{theorem}\label{sconthmain1}
If $(e_n)$ is weakly $L$-invariant then $|t^{s}_{(e_n)}{(\mathbb{T})}|=\mathfrak{c}$.
\end{theorem}
\begin{proof}
Assume that  $(e_n)$ is weakly $L$-invariant. Then there exists an infinite $A\subseteq \mathbb{N}$ such that $d(L(A-m))=0$ for each $m\in\mathbb{N}$. Let us write $A=\{u_1+1<u_2+1<...<u_k+1<...\}$.
Consider $x\in\mathbb{T}$ such that
$$
supp(x)\subseteq A.
$$
Let $0<\epsilon<\frac{1}{2}$ be given. Choose $m\in\mathbb{N}$ such that $\frac{1}{2^{m-2}}<\epsilon$. Set
$$
B=[1,n_{u_m-m+1}-1]\cup\bigcup_{j=m}^{\infty} [n_{u_j-m+1},n_{u_j+1}-1].
$$
Since $d(\displaystyle\bigcup_{i=0}^{m-1} L(A-i))=0$ and $B=^* \displaystyle \bigcup_{i=0}^{m-1} L(A-i)$, we can conclude that $d(B)=0$.
Note that, for sufficiently large $j\in \mathbb{N}\setminus B$, we have $j=n_{k}+i-1$ with $i\in[1,n_{k+1}-n_k]$ where $k\notin \displaystyle\bigcup_{j=m}^{\infty}[u_j-m+1,u_{j}+1]$, i.e., $k,k+1,...,k+m-1\notin supp(x)$.  Note that $e_jx=e_{(n_{k}+i-1)}x=r^{i}_{k+1}a_kx$ where $1\leq r^{i}_{k+1}\leq b_{k+1}$.
Subsequently, we obtain
\begin{eqnarray*}
    r^{i}_{k+1}\{a_kx\}&\leq &r^{i}_{k+1}a_k\sum_{i=k+m}^{\infty}\frac{c_i}{a_i}\\
              &\leq & \frac{r^{i}_{k+1}a_k}{a_{k+m-1}} \leq \frac{r^{i}_{k+1}}{b_{k+1}}\frac{a_{k+1}}{a_{k+m-1}}\leq \frac{1}{2^{m-2}}<\epsilon.
\end{eqnarray*}
Note that in view of Eq (\ref{eqr}), we can write $\{r^{i}_{k+1}a_kx\}=r^{i}_{k+1}\{a_kx\}$. Then, there exists a natural number $j_\varepsilon$ such that for any $j\in \N\setminus B$ and $j\geq j_\varepsilon$, we have
\begin{align*}
    \{e_jx\}=\{r^{i}_{k+1}a_kx\}=r^{i}_{k+1}\{a_kx\}<\epsilon.
\end{align*}
This entails that $x\in t^{s}_{(e_n)}{(\mathbb{T})}$. Since $A$ is infinite there exists $\mathfrak{c}$ many $x\in\T$ such that $supp(x)\subseteq A$. Thus, $t_{(e_n)}^{s}(\mathbb{T})$ contains $\mathfrak{c}$ many elements.
\end{proof}
\begin{theorem}
 If $(e_n)$ is weakly $L$-invariant then $|t_{(e_n)}^{s}(\mathbb{T})\setminus t_{(e_n)}(\mathbb{T})|=\c$.
\end{theorem}
\begin{proof}
  Assume that $(e_n)$ is any weakly $L$-invariant arithmetic-type sequence corresponding to the arithmetic sequence $(a_n)$. Then one can find an infinite set $A\subset \N$ such that $d(L(A-m))=0$ for each $m\in\N$.  Pick any $x \in \mathbb{T}$ with infinite support such that
\begin{equation}\label{eqnddiff}
supp(x) \subseteq A~\mbox{ and }~
c_n = \left\lfloor \frac{b_n}{2} \right\rfloor~\mbox{for each }~n \in \mathbb{N}.
\end{equation}
Since
$supp(x) \subseteq A$ and $d(L(A-m))=0$ for each $m\in\N$, Theorem \ref{sconthmain1} entails that $x\in t_{(e_n)}^{s}(\mathbb{T})$.

Next, we will show that $x \notin t_{(e_n)}(\mathbb{T})$.
Suppose, to the contrary, that $x \in t_{(e_n)}(\mathbb{T})$.
Observe that $ t_{(e_n)}(\mathbb{T})\subseteq t_{(a_n)}(\mathbb{T})$, since
$(a_n)$ is a subsequence of $(e_n)$. Thus, it follows that if $x \notin t_{(a_n)}(\mathbb{T})$, then  $x \notin t_{(e_n)}(\mathbb{T})$.  Now consider the following cases:
\begin{itemize}
    \item First, assume that $supp(x)$ is $b$-bounded. Since $x\in \mathbb{T}$ such that $supp(x)$ is an infinite $b$-bounded subset of $\N$, \cite[Corollary 3.2]{DI1} entails that $x\notin t_{(a_n)}(\mathbb{T})$. This ensures that $x\notin t_{(e_n)}(\mathbb{T})$.
 \item Next, suppose that  $supp(x)$ is not $b$-bounded.\\
Then there exists  $B\subseteq supp(x)$ such that $B$ is $b$-divergent. Since $\displaystyle\lim_{\substack{n\to\infty \\ n\in B}}\frac{c_n}{b_n}=\frac{1}{2} (\neq 0)$, \cite[Theorem 2.3(b)]{DI1} entails that $x\notin t_{(a_n)}(\mathbb{T})$. Thus,  $x\notin t_{(e_n)}(\mathbb{T})$.
\end{itemize}
So, in either case, we obtain $x\in t_{(e_n)}^{s}(\mathbb{T})\setminus t_{(e_n)}(\mathbb{T})$. Since there are $\c$ many $x\in\T$ satisfying Eq. (\ref{eqnddiff}), we can conclude that $|t_{(e_n)}^{s}(\mathbb{T})\setminus t_{(e_n)}(\mathbb{T})|=\c$.
\end{proof}
\begin{lemma}\label{lemmasupport}
Let $(e_n)$ be an arithmetic-type sequence corresponding to the arithmetic sequence $(a_n)$. If $x\in\mathbb{T}$ satisfies $\displaystyle\lim_{n\to\infty}\frac{c_n}{b_n}=0$ and $d(L(supp(x)))=0$, then $x\in t^{s}_{(e_n)}(\mathbb{T})$.
\end{lemma}
\begin{proof}
 Note that $x \in t^{s}_{(e_n)}(\mathbb{T})$ holds vacuously whenever $supp(x)$ is finite. Thus, assume that $x\in\T$ such that $supp(x)$ is infinite and satisfies the aforesaid properties. Now, the condition $\displaystyle\lim_{n\in supp(x)}\frac{c_n}{b_n}=0$ yields that $supp(x)$ is $b$-divergent. Choose an arbitrary $j\in \mathbb{N}\setminus L(supp (x))$. Then there exists $k\in\mathbb{N}\setminus\supp(x)$ such that $j=n_{k-1}+i-1$ for some $i\in[1,n_{k}-n_{k-1}]$. With this, we can write
 $$
 e_jx=e_{(n_{k-1}+i-1)}x=r^i_{k}a_kx~\mbox{ for some ~$1\leq r^i_{k}\leq b_k$}.
 $$
By Eq. (\ref{eqlemain1}), for all $t\geq 0$, we get
$$
\{a_{k}x\}=\frac{c_{k+1}}{b_{k+1}}+\frac{c_{k+2}}{b_{k+1}b_{k+2}}+\ldots+\frac{c_{k+t+1}}{b_{k+1}\ldots b_{k+t+1}}+\frac{\{a_{k+t+1}x\}}{b_{k+1}\ldots b_{k+t+1}}.
$$
For each $k\in\N$, let $l_k$ denote the smallest positive integer such that $k+l_k\in supp(x)$.
Consequently, we obtain
$$
\{a_{k}x\}=\frac{c_{k+l_k}}{b_{k+1}\ldots b_{k+l_k}}+\frac{\{a_{k+l_k}x\}}{b_{k+1}\ldots b_{k+l_k}} \leq \frac{c_{k+l_k}+1}{b_{k+1}\ldots b_{k+l_k}}.
$$
Since $supp(x)$ is $b$-divergent with $\displaystyle\lim_{n\to\infty}\frac{c_n}{b_n}=0$ and $k+l_k\in supp(x)$, it follows that
$$
\displaystyle\lim_{k\to\infty}\{a_{k}x\}=0.
$$
Again, by Eq. (\ref{eqlemain1}) with $t=0$, we get that
 \begin{align}
     \{a_{k-1}x\}=&\left(\frac{c_k}{b_k}+\frac{\{a_kx\}}{b_k}\right)\nonumber\\
     =~&\frac{\{a_kx\}}{b_k}~\mbox{ (since $c_k=0$)}\label{eqnlemma1}.
 \end{align}
 Moreover, since $1\leq r^i_{k}<b_k$ for each $k\in\mathbb{N}$, and
 $\displaystyle\lim_{k\to\infty}\{a_{k}x\}=0$, Eq (\ref{eqnlemma1}) ensures that
 \begin{equation*}
   \displaystyle\lim_{\begin{smallmatrix} k \to \infty & \\k\notin supp(x) \end{smallmatrix}}r^i_{k}\{a_{k-1}x\}=\displaystyle\lim_{\begin{smallmatrix} k \to \infty & \\k\notin supp(x) \end{smallmatrix}}\frac{r^i_{k}}{b_k}\{a_{k}x\}=0.
 \end{equation*}
 Now, Eq. (\ref{eqr}) entails that
 \begin{align*}
      &\displaystyle\lim_{\begin{smallmatrix} k \to \infty & \\k\notin supp(x) \end{smallmatrix}}\{r^i_{k}a_{k-1}x\}=0\\
\Rightarrow    &\displaystyle\lim_{\begin{smallmatrix} k \to \infty & \\k\notin supp(x) \end{smallmatrix}}\{e_{(n_{k-1}+i-1)}x\}=0\\
   \Rightarrow &\displaystyle\lim_{\begin{smallmatrix} j \to \infty & \\j\notin L(supp(x)) \end{smallmatrix}}\{e_jx\}=0.
 \end{align*}
 Since $d(L(supp(x)))=0$, we deduce that $x\in t_{(e_n)}^s(\mathbb{T})$.
 \end{proof}

 \begin{lemma}\label{divergentliftingzero}
   Let $(e_n)$ be an arithmetic-type sequence associated with the arithmetic sequence $(a_n)$. If there exists a $b$-divergent set $D$ such that $d(L(D))=0$ then $|t^{s}_{(e_n)}(\mathbb{T})|=\mathfrak{c}$.
 \end{lemma}
\begin{proof}
Let $D\subset\mathbb{N}$ be a $b$-divergent set such that $d(L(D))=0$. Choose any $x\in\mathbb{T}$ such that

\begin{equation}\label{eqndivlifzero}
  supp(x)\subseteq D~\mbox{ and }~\lim_{n\in D}\frac{c_n}{b_n}=0.
\end{equation}
Observe that $d(L(supp(x)))=0$ and $\displaystyle\lim_{n\in supp(x)}\frac{c_n}{b_n}=0$. Then, by Lemma \ref{lemmasupport}, we obtain $x\in t^{s}_{(e_n)}(\mathbb{T})$.
Since there are $\c$ many $x\in\T$ satisfying Eq. (\ref{eqndivlifzero}), we can conclude that $|t^{s}_{(e_n)}(\mathbb{T})|=\mathfrak{c}$.
\end{proof}
\begin{lemma}\label{divergentliftinggrtzero}
Let $(e_n)$ be an arithmetic-type sequence associated with the arithmetic sequence $(a_n)$ with $e_{n_k}=a_k$ for each $k\in\N\cup\{0\}$. If there exists an infinite $B\subseteq\N$ such that $n_k=2n_{k-1}-n_0$ for each $k\in B$ then for each infinite $A\subseteq B$, $\bar{d}(L(A))>0$.
\end{lemma}
\begin{proof}
Let us write $A=\{u_j+1:j\in\N\}$. Therefore, observe that
 \begin{eqnarray*}
 \bar{d}(L(A)) &=& \limsup_{n\to\infty}\frac{|L(A)\cap[1,n]|}{n} \\
 &\geq& \limsup_{j\to\infty}\frac{|L(A)\cap[1,n_{u_j+1}-1]|}{n_{u_j+1}-1} \\
&=& \limsup_{j\to\infty}\frac{|\displaystyle\bigcup_{k\in\N}[n_{u_k},n_{{u_k}+1}-1]\cap[1,n_{u_j+1}-1]|}{n_{u_j+1}-1} \\
 &=& \limsup_{j\to\infty}\frac{\sum\limits_{k=1}^{j}|[n_{u_k},n_{{u_k}+1}-1]|}{n_{u_j+1}-1}\\
 &=& \limsup_{j\to\infty}\frac{\sum\limits_{k=1}^{j}(n_{{u_k}+1}-n_{u_k})}{n_{u_j+1}-1}\\
 &\geq& \limsup_{j\to\infty} \frac{n_{{u_j}+1}-n_{u_j}}{(n_{u_j+1}-n_{u_j})+(n_{u_j}-n_0)}\\
&=& \limsup_{j\to\infty}\frac{n_{{u_j}+1}-n_{u_j}}{2(n_{u_j+1}-n_{u_j})}=\frac{1}{2} > 0,
\end{eqnarray*}
since $u_j+1\in A\subseteq B$ and $\ n_{u_j+1}=2n_{u_j}-n_0$ for each $j\in \N$.
\end{proof}
\begin{theorem}\label{thmdifferencec}
There exists an arithmetic-type sequence $(e_n)$ which is strongly non $L$-invariant and not weakly $L$-invariant but $|t^{s}_{(e_n)}(\mathbb{T})\setminus t_{(e_n)}(\mathbb{T})|=\mathfrak{c}$.
\end{theorem}
\begin{proof}
 Let us set $b_n=2n$ for each $n\in \N$. Now, consider the associated arithmetic sequence $(a_n)$, defined by
 $$
 a_0=1~\mbox{ and }~a_{n}=b_{n}a_{n-1}~\mbox{ for each }n\in\N.
 $$
 We construct an arithmetic-type sequence $(e_n)$ corresponding to $(a_n)$ with $e_{n_k}=a_k$ such that
 \begin{equation}\label{eqndefinenk}
   n_k=2n_{k-1}-n_0~\mbox{ for each } k\in\N.
 \end{equation}
 Moreover, for each $i\in[1,n_k-n_{k-1}]$, we define
 \begin{equation}\label{eqndefineen}
     e_{(n_{k-1}+i-1)}=r_k^{i}a_{k-1}~\mbox{ where }r_k^{i}\in(\{1\}\cup\{2j:j\in\N\})\cap [1,b_k-1].
 \end{equation}
 Note that, in view of Lemma \ref{divergentliftinggrtzero} and Eq. (\ref{eqndefinenk}), we deduce that $(e_n)$ is strongly non $L$-invariant and not weakly $L$-invariant. Furthermore, Eq. (\ref{eqndefineen}) entails that
$$
n_k-n_{k-1}=\frac{b_k}{2}~\mbox{ for each } k\in\N.
$$
 Thus, $\frac{b_k}{n_k-n_{k-1}}=2$, which is bounded. Consequently, by \cite[Corollary 3.8]{GD2}, it follows that $t_{(e_n)}(\mathbb{T})$ is countable.

 Next, we show that $|t^s_{(e_n)}(\mathbb{T})|=\c$. Let $A=\{l_1<l_2<...<l_k<...\}\subset\N$ be such that $|l_{k+1}-l_{k}|\to\infty$ as $k\to\infty$.
 Choose  $x\in\T$ with $supp(x)$ being infinite such that
\begin{equation}\label{eqnformingelements}
 supp(x)\subseteq A~\mbox{ and }~c_n=\frac{b_n}{2}~\mbox{ for each}~n\in supp(x).
\end{equation}
Since $n_k-n_{k-1}=n_{k-1}-n_0\to\infty$, we have $d(\{n_k:k\in\N\})=0$.
Note that, for each
 $j\in\N\setminus \{n_{k}:k\in\N\}$, we have $j=n_{k-1}+i-1$ for some $k\in\N$ and some $2\leq i\leq n_k-n_{k-1}$. Then, by the way $(e_n)$ is constructed, we can write
 $$
 e_{j}x=e_{(n_{k-1}+i-1)}x=r^i_{k}a_{k-1}x~\mbox{ where }r^i_{k}\in2\N\cap[1,b_k-1].
 $$
 Note that, for each $k\in \N$, either $k\in supp(x)$ or there exists $m\geq 2$ such that $k+m\in\supp(x)$. If $k\in supp(x)$, then by Eq. (\ref{eqsum}) and Eq. (\ref{eqlemain1}) we obtain
 \begin{align}\label{eqn18}
    &\frac{c_k}{b_k}\leq \{a_{k-1}x\}\leq \frac{c_k}{b_k}+\frac{c_{k+m}+1}{b_kb_{k+1}\cdots b_{k+m}}\nonumber\\
    \Rightarrow~& \frac{1}{2}\leq \{a_{k-1}x\}\leq \frac{1}{2}+\frac{1}{b_kb_{k+1}\cdots b_{k+m-1}}~\mbox{ (since }c_i\leq b_i-1~\mbox{for each }i\in\N)\nonumber\\
    \Rightarrow~& \frac{r^i_{k}}{2}\leq r^i_{k}\{a_{k-1}x\}\leq \frac{r^i_{k}}{2}+\frac{r^i_{k}}{b_k}\cdot\frac{1}{b_{k+1}\cdots b_{k+m-1}}\nonumber\\
    \Rightarrow~& \frac{r^i_{k}}{2}\leq r^i_{k}\{a_{k-1}x\}\leq \frac{r^i_{k}}{2}+\frac{1}{b_{k+1}\cdots b_{k+m-1}}~\mbox{(since }~1\leq r^i_{k}<b_k).
 \end{align}
 If $k\notin supp(x)$, then again using Eq. (\ref{eqsum}) and Eq. (\ref{eqlemain1}), we can write
 \begin{align}\label{eqn19}
    &\{a_{k-1}x\}\leq \frac{c_{k+m}+1}{b_kb_{k+1}\cdots b_{k+m}}\nonumber\\
    \Rightarrow~& r^i_{k}\{a_{k-1}x\}\leq \frac{r^i_{k}}{b_k}\cdot\frac{1}{b_{k+1}\cdots b_{k+m-1}}\nonumber\\
    \Rightarrow~& r^i_{k}\{a_{k-1}x\}\leq \frac{1}{b_{k+1}\cdots b_{k+m-1}}.
\end{align}
Let $\epsilon>0$ be arbitrary.
 Since $(a_n)$ is $b$-divergent and $r^i_{k}\in 2\N$, it follows that for sufficiently large $k\in \N$, by Eq. (\ref{eqn18}) and Eq. (\ref{eqn19}), we obtain
 $$
 \|r^i_{k}\{a_{k-1}x\}\|<\epsilon.
 $$
Consequently, there exists $j_{\varepsilon}\in\N$  such that for all $j\in\N\setminus \{n_{k}:k\in\N\}$ with $j\geq j_{\varepsilon}$, we have
$$
\|e_jx\|=\|e_{(n_{k-1}+i-1)}x\|=\|r^i_{k}a_{k-1}x\|= \|r^i_{k}\{a_{k-1}x\}\|<\epsilon.
$$
Since $d(\{n_{k}:k\in\N\})=0$, it follows that $x\in t^s_{(e_n)}(\mathbb{T})$.
Observe that there are $\c$ many elements $x\in\T$ satisfying Eq. (\ref{eqnformingelements}). This ensures that  $|t^s_{(e_n)}(\mathbb{T})|=\c$. Since $t_{(e_n)}(\mathbb{T})$ is countable, we can conclude that $|t^{s}_{(e_n)}(\mathbb{T})\setminus t_{(e_n)}(\mathbb{T})|=\mathfrak{c}$.
\end{proof}
From Theorem \ref{thmdifferencec}, it is natural to pose the following question:
\begin{problem}\label{pr1}
 Under what conditions does $t^s_{(e_n)}(\mathbb{T})=t_{(e_n)}(\mathbb{T})$ hold for arithmetic-type sequences $(e_n)$?
\end{problem}
In contrast to the case of $(d_n)$, the subgroup $t_{(e_n)}(\mathbb{T})$ need not be countable for arithmetic-type sequences $(e_n)$.
Consequently, even a positive solution to Problem \ref{pr1} does not imply that $t^s_{(e_n)}(\mathbb{T})$ is countable. So we formulate the following problem:
\begin{problem}
 When  $t^s_{(e_n)}(\mathbb{T})$ is countable?
\end{problem}


\begin{thebibliography}{100}

\bibitem{A1} J. Arbault, Sur l'ensemble de convergence absolue d'une s\' erie trigonom\' etrique., Bull. Soc. Math. Fr. 80 (1952),  253--317.



\bibitem{AT} A. Arhangel’skii , M. Tkachenko, Topological Groups and Related Structures: An Introduction to Topological Algebra, Atlantis Press, Paris, 2008.

\bibitem{AD} L. Außenhofer, D. Dikranjan, Locally quasi-convex compatible topologies on locally compact abelian groups, Mathematische Zeitschrift, 296 (2020), 325-351.


\bibitem{BDBW} G. Barbieri, D. Dikranjan, A. Giordano Bruno, H. Weber, Dirichlet sets vs characterized subgroups, Topol. Appl. 231, 50--76 (2017).

\bibitem{BDK}   M. Balcerzak, K. Dems, A. Komisarski, Statistical convergence and ideal convergence for sequences of functions, J. Math. Anal. Appl., 328(1) (2007), 715--729.


\bibitem{BDMW1} G. Babieri, D. Dikranjan,  C, Milan, H. Weber, Topological torsion related to some recursive sequences of integers, Math. Nachr. 281(7) (2008), 930--950.




\bibitem{BDS} A. B\' \i r\' o, J.M. Deshouillers, V.T.  S\' os, Good approximation and characterization of subgroups of $\R/\Z$, Studia Sci. Math. Hungar. 38 (2001), 97--113.

\bibitem{B1} J.-P. Borel,  Sous-groupes de $\R$ li\' es \' a  la r\' epartition modulo 1 de suites, Ann. Fac. Sci. Toulouse Math. 5(3--4) (1983), 217--235.

\bibitem{B2} J.-P. Borel,  Sur certains sous-groupes de R li\' es \' a la suite des factorielles, Colloq. Math. 62(1) (1991), 21--30.


\bibitem{Bu1} R. C. Buck, The measure theoretic approach to density, Amer. J. Math. 68 (1946), 560--580.



\bibitem{BuKR} L. Bukovsk\' y,  N. Kholshchevnikova, N.N., M. Repick\'  y, Thin sets in harmonic analysis and infinite combinatorics, Real Anal. Exch. 20 (1994/1995), 454--509.

\bibitem{CKG} J. Connor, V. Kadets, M. Ganichev,  A Characterization of Banach Spaces with Separable Duals via Weak Statistical Convergence, J. Math. Anal. Appl.,  244 (2000), 251--261.

\bibitem{DG2} P. Das, A. Ghosh, Solution of a general version of Armacost's problem on topologically torsion elements, Acta Math. Hungar., 164(1) (2021), 243--264.

\bibitem{DG} P. Das, A. Ghosh, On a new class of trigonometric thin sets extending Arbault sets, Bul. Sci. Math., 179 (2022), 103157.

\bibitem{DG8} P. Das, A. Ghosh, Characterized subgroups related to some non-arithmetic sequence of integers, Mediterr. J. Math., 21 (2024), 164.

\bibitem{DG9} P. Das, A. Ghosh, Statistically characterized subgroups related to some non-arithmetic sequence of integers,  Expo. Math., 43 (2025), 125653.

\bibitem{DGT1}  P. Das, A. Ghosh, T. Aziz, Topologically torsion elements of the circle group for arithmetic-type sequences, Bul. Sci. Math., 199 (2025) 103580.

\bibitem{DGT2}  P. Das, A. Ghosh, T. Aziz, Statistically characterized subgroups related to some non-arithmetic sequence of integers II (characterization for countable subgroups), J. Lond. Math. Soc., 112-6 (2025), e70365.

\bibitem{D_NYC} D. Dikranjan, Topologically torsion elements of topological groups, Topol. Proc. 26 (2001--2002), 505--532.


\bibitem{DDB} D. Dikranjan, P. Das, K. Bose, Statistically characterized subgroups of the circle, Fund. Math., 249 (2020), 185-209.



\bibitem{DI1} D. Dikranjan and D. Impieri, Topologically torsion elements of the circle group, Commun. Algebra, 42 (2014),  600--614.



\bibitem{DK}   D. Dikranjan, K. Kunen, Characterizing countable subgroups of compact abelian groups, J. Pure Appl. Algebra 208 (2007), 285--291.


\bibitem{DPS} D. Dikranjan, Iv. Prodanov and L. Stoyanov, Topological Groups: Characters, Dualities and Minimal Group Topologies, Pure and Applied Mathematics, Marcel Dekker Inc. New York (1989).

\bibitem{DR} D. Dikranjan, R. Di Santo, Answer to Armacost's quest on topologically torsion elements of the circle group, Comm. Algebra, 32 (2004), 133--146 .

\bibitem{DDBH} D. Dikranjan, R. Di Santo, A. Giordano Bruno and H. Weber, An introduction to $\mathcal{I}$-characterized subgroups of the circle, Topology Appl., 2025, to appear.

\bibitem{MK}  G. Di Maio, Lj.D.R. Ko$\check{c}$inac, Statistical convergence in topology, Topology Appl. 156 (2008), 28--45.


\bibitem{El}  P. Elia\v s, A classification of trigonometrical thin sets and their interrelations, Proc. Amer. Math. Soc. 125(4) (1997), 1111--1121.

\bibitem{F} H. Fast, Sur la convergence statistique, Colloq. Math. 2 (1951), 241--244.

\bibitem{FGT} A. Faisant, G. Grekos, V. Toma, On the statistical variation of sequences, J. Math. Anal. Appl., 306 (2) (2005), 432--439.

\bibitem{Fr} J. A. Fridy, On statistical convergence, Analysis 5(4) (1985), 301--313.



\bibitem{GD1} A. Ghosh, P. Das, Characterizing infinite torsion subgroups of the circle via arithmetic-type sequences, Bull. Lond. Math. Soc. 57(12) (2025), 3824--3840.

\bibitem{GD2} A. Ghosh, P. Das, Characterizing infinite torsion subgroups of the circle via arithmetic-type sequences II, submitted.



\bibitem{I} D. Impieri, Characterized subgroups, PhD Thesis, July 2015.

\bibitem{Ka} S. Kahane, Antistable classes of thin sets in harmonic analysis, Illinois J. Math. 37 (1993), 186-223.






\bibitem{R} L. Robertson, Connectivity, divisibility, and torsion, Trans. Amer. Math. Soc. 128 (1967) 482--505.

\bibitem{S} T. \v Sal\' at, On statistically convergent sequences of real numbers, Mathematica Slovaca, 30(2) (1980), 139--150.


\bibitem{DGDis} R. Di Santo,  D. Dikranjan and A. Giordano Bruno, Characterized subgroups of the circle group, Ric. Mat. 67(2) (2018), 625--655.


\bibitem{St} H. Steinhaus, Sur la convergence ordinaire et la convergence asymptotique, Colloq. Math. 2 (1951), 73--74.

\bibitem{W} H. Weyl, \" Uber die Gleichverteilung von Zahlen mod. Eins., Math. Ann. 77(3) (1916), 313--352.

\bibitem{Wi} R. Winkler, Ergodic group rotations, Hartman sets and Kronecker sequences (dedicated to Edmund Hlawka on the occasion of his $85^{th}$ birthday), Monatsh. Math. 135, No.4, 333-343 (2002).

\bibitem{Z} A. Zygmund, Trigonometric Series, vols. I,II, Cambridge University Press, Cambridge, New York, Melbourne (1977).
\end{thebibliography}
\end{document}